\documentclass[12pt,leqno]{amsart}  %10 pt minimum, 12 recommended

\usepackage{calc}
\usepackage[left=1in,right=1in,top=1in,bottom=1in]{geometry}  %I had to print .pdf using Adobe Acrobat to keep these margins in a printed copy!!
\usepackage{graphicx}
\usepackage{amsmath,amssymb,epsfig,mathtools}
\usepackage{amsthm}
\usepackage{enumerate}
\usepackage{caption}
\usepackage{tikz}
\makeatletter
\renewcommand*\env@matrix[1][*\c@MaxMatrixCols c]{%
  \hskip -\arraycolsep
  \let\@ifnextchar\new@ifnextchar
  \array{#1}}
\makeatother

\newcommand{\val}{{\mathrm{val}}}

\newcommand{\p}{{ \mathfrak{p} }}

\newcommand{\Deltam}{{ \Delta_\mathrm{min} }}
\newcommand{\vm}{{ v_\mathrm{min} }}

\newcommand{\FF}{\mathbb{F}}

\newcommand{\QQ}{\mathbb{Q}}

\newcommand{\cO}{\mathcal{O}}

\newtheorem{thm}{\bf{Theorem}}
\newtheorem{theorem}{\bf{Theorem}}[section]

\newtheorem{corollary}[theorem]{\bf{Corollary}}

\newtheorem{proposition}[theorem]{\bf{Proposition}}

\newtheorem{lemma}[theorem]{\bf{Lemma}}

\theoremstyle{definition}

\newtheorem{remark}[theorem]{\bf{Remark}}

\title{On a local invariant of elliptic curves with a $p$-isogeny}
\author{Matthew Gealy}
\address[M.~Gealy]{Center for Communications Research, 4320 Westerra Court, San Diego, CA 92121}
\email{mtgealy@ccrwest.org}

\author{Zev Klagsbrun}
\address[Z.~Klagsbrun]{Center for Communications Research, 4320 Westerra Court, San Diego, CA 92121}
\email{zdklags@ccrwest.org}

\begin{document}
\bibliographystyle{alpha}

\begin{abstract}
An elliptic curve $E$ defined over a $p$-adic field $K$ with a $p$-isogeny $\phi:E\rightarrow E^\prime$ comes equipped with an invariant $\alpha_{\phi/K}$ that measures the valuation of the leading term of the formal group homomorphism 
%An elliptic curve $E$ with a $p$-isogeny $\phi:E\rightarrow E^\prime$ defined over a
%local field $K$ having residual characteristic $p$ comes equipped with an invariant $\alpha_{\phi/K}$ that measures the valuation of the leading term of the formal group homomorphism 
$\Phi:\hat E \rightarrow \hat E^\prime$.
%$\Phi:\hat E(\mathfrak{m}_{O_K}) \rightarrow \hat E^\prime(\mathfrak{m}_{O_K})$.
We prove that if $K/\QQ_p$ 
is unramified and $E$ has additive, potentially supersingular reduction, then $\alpha_{\phi/K}$ is determined by the number of distinct geometric components on the special fibers of the minimal proper regular models of $E$ and $E^\prime$.% over $K^{ur}$.
\end{abstract}
\maketitle

\section{Introduction}

Let $K$ be a 
%local field having residual characteristic $p$ 
$p$-adic field and $E$ an elliptic curve defined over a $K$ 
%local field $K$ 
having a cyclic $p$-isogeny $\phi:E\rightarrow E^\prime$. The isogeny $\phi$ induces a homomorphism 
% on the formal groups
$\Phi:\hat E \rightarrow \hat E^\prime$ where $\hat E$ and $\hat E^\prime$ are formal groups of $E$ and $E^\prime$ constructed using minimal invariant differentials of $E$ and $E^\prime$. Well known results (see Lemma 4.2 in \cite{DD}, for example) show that $\Phi$ is given by a formal power series $\Phi(T) = a_1 T + a_2T^2 + \ldots$ where $a_1 = \frac{\phi^* \omega^\prime}{\omega}  \times u$, for some unit $u \in \cO_K^\times$.

Since minimal differentials are only unique up to units, the only information intrinsic to the curve is the valuation of $\frac{\phi^* \omega^\prime}{\omega}$. 
%
%Following \cite{DD}, we define a quantity $\alpha_{\phi/F}$ by  $\alpha_{\phi/F} = \left | \frac{\phi^* \omega^\prime}{\omega} \right |_F^{-1}$, where $| \cdot |_F$ is the normalized valuation on $F$.
Let $\alpha_{\phi/K} = \left | \frac{\phi^* \omega^\prime}{\omega} \right |_K^{-1}$, where $| \cdot |_K$ is the normalized valuation on $K$. The quantity $\alpha_{\phi/K}$ plays an important role in descent via isogeny \cite{Schaefer}
% the size of $E^\prime(K)/\phi(E(K))$ is given by \linebreak  $\alpha_{\phi/K} \cdot |E(K)[\phi]| \cdot \frac{c^\prime}{c}$, where $c$ and $c^\prime$ are the Tamagawa numbers of $E/K$ and $E^\prime/K$ respectively~\cite{Schaefer}
and also in 
%Understanding  $\alpha_{\phi/K}$ and its behavior under quadratic twist also plays a key role in 
the recent work of Bhargava, Klagsbrun, Lemke Oliver, and Shnidman on the distribution of $\phi$-Selmer groups in quadratic twist families of elliptic curves with a cyclic $3$-isogeny \cite{BKLS}.

Thanks to the work of the Dokchitsers in \cite{DD}, the behavior of $\alpha_{\phi/K}$ is well-understood when $E$ has either (potential) multiplicative or (potential) good ordinary reduction, yet little is known about the cases where $E$ has good supersingular or additive potentially supersingular reduction.

We obtain the following new result, which completely characterizes $\alpha_{\phi/K}$ in the case where $K/\QQ_p$ is unramified.

\begin{thm}
\label{thm:mainthm}
Suppose that $K/\QQ_p$ is unramified and $E/K$ has additive, potentially supersingular reduction. Let $m(E/K)$ and $m(E^\prime/K)$ be the number of distinct geometric components on the special fibers of the minimal proper regular models of $E$ and $E^\prime$ respectively. Then $m(E/K) \ne m(E^\prime/K)$ and
$$\alpha_{\phi/K} = \left \{ \begin{matrix}
1 & \text{ if } m(E/K) < m(E^\prime/K) \\
p^{\deg(K)} & \text{ if } m(E/K) > m(E^\prime/K)
\end{matrix}
\right . .
$$
Equivalently, $\vm(E/K) \ne \vm(E^\prime/K)$ and
$$\alpha_{\phi/K} = \left \{ \begin{matrix}
1 & \text{ if } \vm(E/K) < \vm(E^\prime/K)\\
p^{\deg K} & \text{ if }\vm(E/K) > \vm(E^\prime/K)
%\N_{K/\QQ_p} (p) & \text{ if }\vm(E/K) > \vm(E^\prime/K)
\end{matrix}
\right . ,
$$
where $\vm(E/K)$ and $\vm(E^\prime/K)$ are the valuations of the discriminants of minimal models of $E$ and $E^\prime$ respectively.  
\end{thm}

\begin{remark}
\label{rem:Ogg}
The equivalence between the two statements in Theorem \ref{thm:mainthm} is a consequence of Ogg's formula (see Section IV.11.1 in \cite{ATAEC}, for example)
$$\vm(E/K) = \val_K(\mathfrak{f}_{E/K}) + m(E/K) - 1$$
where $\mathfrak{f}_{E/K}$ is the conductor of $E/K$.
Since $E/K$ and $E^\prime/K$ have the same conductor, we get $$\vm(E/K) - \vm(E^\prime/K) = m(E/K) - m(E^\prime/K).$$
\end{remark}

\begin{remark}
\label{rem:nogoodss}
As shown in Corollary \ref{cor:noss}, if $E/K$ has good supersingular reduction, then $K/\QQ_p$ 
$K$ 
must be ramified. As a result, Theorem \ref{thm:mainthm} completes the characterization of $\alpha_{\phi/K}$ begun in \cite{DD} in the case where $K/\QQ_p$ is unramified.% extension of $\QQ_p$.
\end{remark}

No part of Theorem \ref{thm:mainthm} remains true if $K/\QQ_p$ is ramified. We do however have the following partial converse. 

\begin{thm}
\label{thm:ramthm}
Suppose that $E/K$ has either good supersingular or additive potentially supersingular reduction. 
\begin{enumerate}[(i)]
\item If $\alpha_{\phi/K} = 1$, then $m(E/K) < m(E^\prime/K)$ and $\vm(E/K)  < \vm(E^\prime/K)$.
\item If $\alpha_{\phi/K} = p^{\deg K} $, then $m(E/K) > m(E^\prime/K)$ and $\vm(E/K)  > \vm(E^\prime/K)$.
\end{enumerate}
\end{thm}

\section*{Notation}
We will use the following notation throughout this paper:
\begin{itemize}
\item $K$ will be a finite extension of $\QQ_p$.%will be a local field of residual characteristic $p$.
\item $E/K$ will be an elliptic curve defined $K$ with a rational $p$-isogeny $\phi:E \rightarrow E^\prime$. The dual isogeny will be denoted $\phi^\prime$.
\item For an extension $F/K$, 
\begin{itemize}
  \item $\p$ will be unique prime ideal of $\cO_F$.
  \item $\FF$ will be the residue class field of $\cO_F$.
  \item $e(F/K)$ will denote the ramification index of $F/K$.
  \item $\val_F$ will denote the normalized additive valuation on $F$.
  \item $E/F$ will denote the base change $E \otimes_K F$.
  \item $\omega_F$ and $\omega_F^\prime$ will denote minimal invariant differentials on $E/F$ and $E^\prime/F$ respectively.
  \item $\Deltam(E/F)$ will denote a minimal discriminant of $E/F$.
  \item $\vm(E/F)$ will denote $\val_F(\Deltam(E/F))$.
  \item $\mathcal{E}_{\cO_F}$ will denote the minimal proper regular model of $E/F$ and $\mathcal{E}_\FF$ will denote the special fiber.% of $\mathcal{E}_{\cO_F}$.
\item $m(E/F)$ will denote the number of distinct irreducible components of $\mathcal{E}_\FF$ %$\widetilde{\mathcal{E}_K}$
defined over $\overline{\FF}$.%K^{ur}$.
\end{itemize}
\end{itemize}
Theorem \ref{thm:mainthm} is proved by base-changing $E/K$ to a field $L$ over which it obtains good reduction. We therefore explicitly include the base fields in our notation to avoid any confusion, though we allow ourselves to abandon this convention when there is no ambiguity about the field.

\section{Differentials}

Let $\omega_F$ and $\omega_F^\prime$ be minimal invariant differentials --- that is, invariant differentials on minimal Weierstrass models --- of $E/F$ and $E^\prime/F$ respectively.
%
%The global differentials on $E/F$ and $E^\prime/F$ form one-dimensional $F$-vector spaces generated by $\omega_F$ and $\omega_F^\prime$ respectively. The quotient $\dfrac{\phi^* \omega_F^\prime}{\omega_F}$ measures how close the pullback of $\omega_F^\prime$ under $\phi$ comes to generating the global differentials on $E/F$.
We begin with some basic results about the quotient $\dfrac{\phi^* \omega_F^\prime}{\omega_F}$.

\begin{proposition}
\label{prop:sumtovalp}\text{ }
\begin{enumerate}[(i)]
\item Both $\dfrac{\phi^* \omega_F^\prime}{\omega_F}$ and $\dfrac{{\phi^\prime}^* \omega_F}{\omega_F^\prime}$ are in $\cO_F$.
\item We have $\val_F \left (\dfrac{\phi^* \omega_F^\prime}{\omega_F} \right ) + \val_F\left( \dfrac{{\phi^\prime}^* \omega_F}{\omega_F^\prime} \right ) = \val_F (p)$.
\end{enumerate}
\end{proposition}
\begin{proof}
Part (i) is the same as part (1) of Lemma 4.2 in \cite{DD}. To see (ii), we apply the argument from the proof of Lemma 4.3 in \cite{DD}: since $\phi^\prime \circ \phi = [p]$, we have
$
\dfrac{\phi^* \omega_F^\prime}{\omega_F}\cdot \dfrac{{\phi^\prime}^* \omega_F}{\omega_F^\prime} = p
$.
Taking valuations gives the result.
\end{proof}

Proposition \ref{prop:sumtovalp} yields two immediate corollaries.

\begin{corollary}
\label{cor:sumtovalp}
%If $\alpha_{\phi/F} = 1$, then $\alpha_{ \phi^\prime/F} = \N_{F/\QQ_p} (p)$.
If $\alpha_{\phi/F} = 1$, then $\alpha_{ \phi^\prime/F} = p^{\deg F}$.
\end{corollary}
\begin{proof}
If $\alpha_{\phi/F} = 1$, then we must have $\val_F \left (\dfrac{\phi^* \omega_F^\prime}{\omega_F} \right ) = 0$. By Proposition \ref{prop:sumtovalp}, we therefore get  $\val_F\left( \dfrac{{\phi^\prime}^* \omega_F}{\omega_F^\prime} \right ) = \val_F (p)$. The result follows since $|p|_F^{-1} = p^{\deg F}$.
\end{proof}

\begin{corollary}
\label{cor:oneeq1}
If $F$ is unramified, then one of $\alpha_{\phi/F}$ and $\alpha_{ \phi^\prime/F}$ is equal to $1$ and the other is equal to $p^{\deg F}$.
\end{corollary}
\begin{proof}
By part (i) of Proposition \ref{prop:sumtovalp}, both $\val_F \left (\dfrac{\phi^* \omega_F^\prime}{\omega_F} \right )$ and $\val_F\left( \dfrac{{\phi^\prime}^* \omega_F}{\omega_F^\prime} \right )$ are non-negative. By part (ii) of Proposition \ref{prop:sumtovalp}, these sum to $\val_F(p)$, which is equal to $1$ since $F$ is unramified. As a result, one of $\val_F \left (\dfrac{\phi^* \omega_F^\prime}{\omega_F} \right )$ and $\val_F\left( \dfrac{{\phi^\prime}^* \omega_F}{\omega_F^\prime} \right )$ is $0$ and the other is $1$.
\end{proof}

\subsection{Minimal proper regular models}

We now turn to the case where $E/F$ has good reduction.

If $E/F$ has good reduction, then a minimal Weierstrass model for $E/F$ defines a minimal proper regular model $\mathcal{E}_{\cO_F}$ for $E/F$. The special fiber $\mathcal{E}_\FF$ is defined by the reduction modulo $\p$ of this minimal Weierstrass model.

The differential $\omega_F$ generates the space of global differentials on $E/F$, which is a one-dimensional $F$-vector space and the space of global differentials on $\mathcal{E}_{\cO_F}$,  which is a rank-one $\cO_F$-module. The reduction of $\omega_F$ modulo $\p$ to the minimal Weierstrass model for $E/F$ is non-trivial and generates the space of differentials on $\mathcal{E}_{\FF}$ as a one-dimensional $\FF$-vector space. We have a similar story for $E^\prime/F$.
%The quotient $\dfrac{\phi^* \omega_F^\prime}{\omega_F}$ measures how close the pullback of $\omega_F^\prime$ under $\phi$ comes to generating the global differentials on $E/F$.

Since $E/F$ has good reduction, the minimal proper regular model for $E/F$ and the Neron mininal model for $E/F$ coincide. As a consequence of the Neron universal mapping property, the isogeny $\phi:E \rightarrow E^\prime$ therefore induces an $\cO_F$-morphism $\phi_{\cO_F}: \mathcal{E}_{\cO_F} \rightarrow  \mathcal{E^\prime}_{\cO_F}$ on minimal proper regular models (see Exercise 4.24 in \cite{ATAEC}, for example). The restriction of $\phi_{\cO_F}$ to the special fiber $\mathcal{E}_{\FF}$ then yields an $\FF$-morphism $\phi_\FF: \mathcal{E}_{\FF} \rightarrow  \mathcal{E^\prime}_{\FF}$.
%
%$\phi_\FF: \widetilde{\mathcal{E}} \rightarrow  \widetilde{\mathcal{E^\prime}}$ 
%$\phi_\FF: \mathcal{E}_{\FF} \rightarrow  \mathcal{E^\prime}_{\FF}$
%on special fibers (see Exercise 4.24 in \cite{ATAEC}).
The invariant $\alpha_{\phi/F}$ measures how far the map $\phi_\FF$ is from being separable.

\begin{lemma}
\label{lem:sep}
If $E/F$ has good reduction, then $\phi_\FF: \mathcal{E}_{\FF} \rightarrow  \mathcal{E^\prime}_{\FF}$ is separable if and only if $\alpha_{\phi/F}~=~1$.
\end{lemma}
\begin{proof}
By Proposition II.4.2(c) in \cite{AEC}, $\phi_\FF$ will be separable if and only if $\phi_\FF^* \left(\omega_F^\prime \pmod{\p} \right) \ne 0$. We therefore wish to show that $\alpha_{\phi/F} = 1$ if and only if $\phi_\FF^* \left(\omega_F^\prime \pmod{\p} \right) \ne 0$.

%final condition that $\phi_\FF^* \left(\omega_F^\prime \pmod{\p} \right) \ne 0$ is equivalent to $\phi_\FF$ being separable, so the result follows.
%Since $E$ has good reduction, $\mathcal{E}_{\cO_F}$ of $E$ is given by an integral Weierstrass model having good reduction. The global differentials on $\mathcal{E}_{\cO_F}$ form a rank one $\cO_F$-module generated by $\omega_F$ and the restriction of $\omega_F$ to $\mathcal{E}_\FF$ given by $\omega_F \pmod{\p}$ generates the global differentials on the special fiber $\mathcal{E}_{\FF}$. We have a similar story for $E^\prime$. As a result, we have
Indeed, by the above discussion, we have 
\begin{multline*}
\alpha_{\phi/F} = 1
\Leftrightarrow  
\dfrac{\phi^* \omega_F^\prime}{\omega_F} \in \cO_F^\times 
\Leftrightarrow  
\phi_{\cO_F}^* \omega_F^\prime \text{ generates the global differentials on } \mathcal{E}_{\cO_F} \\
\Leftrightarrow  
\phi_{\cO_F}^* \omega_F^\prime  \pmod{\p} \text{ generates the global differentials on } \mathcal{E}_\FF\\
\Leftrightarrow  
\phi_{\cO_F}^* \omega_F^\prime  \pmod{\p} \ne 0
\Leftrightarrow  
\phi_\FF^* \left(\omega_F^\prime \pmod{\p} \right) \ne 0.
\end{multline*}
%By Proposition II.4.2(c) in \cite{AEC}, the final condition that $\phi_\FF^* \left(\omega_F^\prime \pmod{\p} \right) \ne 0$ is equivalent to $\phi_\FF$ being separable, so the result follows.
\end{proof}

%
%\begin{corollary}
%\label{cor:oneeq1}
%Suppose that $F/\QQ_p$ is unramified and let $\phi^\prime$ be the dual isogeny of $\phi$. Then exactly one of $\alpha_{\phi/F}$ and $\alpha_{ \phi^\prime/F}$ is equal to $1$ and the other is equal to $\N_{F/\QQ_p} (p)$.
%\end{corollary}
%\begin{proof}
%If $p = 3$, then $\alpha_{\phi,\Q_p}, \alpha_{\hat\phi,\Q_p} \in \Z$ and $\alpha_{\phi,\Q_p} \cdot \alpha_{\hat\phi,\Q_p} = 3$.
%By Lemma 4.2 in \cite{DD}, we have both $\dfrac{\phi^* \omega^\prime}{\omega},\dfrac{{\phi^\prime}^* \omega}{\omega^\prime} \in \cO_F$. We now apply the argument from the proof of Lemma 4.3 in \cite{DD}: since $\phi^\prime \circ \phi = [p]$, we have
%$
%\dfrac{\phi^* \omega^\prime}{\omega}\dfrac{{\phi^\prime}^* \omega}{\omega^\prime} = p.
%$.
%Since $F/\QQ_p$ is unramified, we therefore have $\val_F \left(\dfrac{\phi^* \omega^\prime}{\omega} \right) + \val_F\left(\dfrac{{\phi^\prime}^* \omega}{\omega^\prime}\right) = 1$.
%As both $\dfrac{\phi^* \omega^\prime}{\omega},\dfrac{{\phi^\prime}^* \omega}{\omega^\prime} \in \cO_F$, we therefore get that one of $\val_F \left(\dfrac{\phi^* \omega^\prime}{\omega} \right)$ and  $\val_F\left(\dfrac{{\phi^\prime}^* \omega}{\omega^\prime}\right)$ is equal to zero and the other is equal to one. The result then follows from the definition of $\alpha_{\phi/F}$
%\end{proof}

Combining Corollary \ref{cor:oneeq1} and Lemma \ref{lem:sep}, we then get:
\begin{corollary}
\label{cor:onesep}
If $F/\QQ_p$ 
%$F$
is unramified and $E/F$ has good reduction, then exactly one of $\phi_\FF$ and $\phi_\FF^\prime$ is separable.
\end{corollary}
%\begin{proof
%This follows from Proposition \ref{cor:oneeq1} and the equivalence in Lemma \ref{lem:sep}.
%\end{proof}

In general however, it will not always be the case that one of $\phi_\FF$ and $\phi_\FF^\prime$ is separable.

\begin{proposition}
\label{prop:ssbothbad}
If $E/F$ has good supersingular reduction, then neither $\phi_\FF$ and $\phi_\FF^\prime$ is separable. As a result, neither $\alpha_{\phi/F}$ and $\alpha_{ \phi^\prime/F}$ is equal to $1$.
\end{proposition}
\begin{proof}
Since $E/F$ has supersingular reduction, the map
%$\widetilde{[p]}:\widetilde{\mathcal{E}}  \rightarrow \widetilde{\mathcal{E}}$
$[p]_{\mathcal{E}_\FF}:\mathcal{E}_\FF  \rightarrow \mathcal{E}_\FF$
is purely inseparable (see Theorem V.3.1 in \cite{AEC}, for example). Since $[p]_{\mathcal{E}_\FF} = \phi_\FF^\prime \circ \phi_\FF$, neither $\phi_\FF$ nor $\phi_\FF^\prime$ can be separable. By Lemma \ref{lem:sep}, we therefore get that neither of $\alpha_{\phi/F}$ and $\alpha_{ \phi^\prime/F}$ is equal to $1$.
\end{proof}

\begin{corollary}
\label{cor:noss}
If $E/F$ has good supersingular reduction, then $F/\QQ_p$
%$F$ 
must be ramified.
\end{corollary}
\begin{proof}
If $F/\QQ_p$
%$F$
were unramified, then this would cause a contradiction between Corollary \ref{cor:onesep} and Proposition \ref{prop:ssbothbad}.
\end{proof}

\section{Proofs of Theorems}

The core idea of the proof of Theorem \ref{thm:mainthm} is to examine what happens when we base change $E$ to an extension $L/K$ where $E$ obtains good reduction.
%If $L/K$ is a finite extension, we will let $E/L$ denote the base change $E \otimes_K L$.
%We will denote the Neron differential of $E/K$ by $\omega_K$ and the Neron differential of $E/L$ by $\omega_L$

\begin{lemma}
\label{lem:diffchangetoL}
We have $ \omega_L = u \cdot \omega_K$, where $$\val_L (u) = \frac{\val_{L} (\Deltam(E/K)) - \val_{L}( \Deltam(E/L))}{12}.$$  
\end{lemma}
\begin{proof}
We may assume that $\omega_K$ is given by the invariant differential on a minimal Weierstrass model of $E/K$. A minimal Weierstrass model for $E/L$ is then obtained via a coordinate change $(x,y) \mapsto \left( \frac{(x-r)}{u^2}, \frac{y - s(x-r) + t}{u^3} \right)$ for appropriate values of $u$, $s$, $r$, and $t$ in  $L$.

The differential $\omega_L$ is then the given by the invariant differential on this minimal model, which is equal to $u \cdot \omega_K$. The relationship between $\Deltam(E/K)$ and $\Deltam(E/L)$ is given by $u^{12}\Deltam(E/L) = \Deltam(E/K)$, so the result follows from taking valuations.
\end{proof}

\begin{corollary}
\label{cor:diffchangetogood}
If $E/K$ has obtains good reduction over $L$, then  $ \omega_L = u \cdot \omega_K$ for some $u \in \cO_L$ with $\val_L (u) = \frac{e(L/K)\vm(E/K)}{12}.$
\end{corollary}
\begin{proof}
Since $E/L$ has good reduction, we have $\val_{L} (\Deltam(E/L)) = 0$. By Lemma \ref{lem:diffchangetoL}, we therefore have $\val_L (u) = \frac{\val_{L} (\Deltam(E/K))}{12}.$ The result about $\val_L(u)$ then follows since $\val_{L}(\pi_K) = e(L/K)$ for any uniformizer $\pi$ of $\cO_K$. As $\val_L(u)$ is non-negative, we have $u \in \cO_L$.
\end{proof}

We are now able to prove Theorem \ref{thm:ramthm}.

\begin{proof}[Proof of Theorem \ref{thm:ramthm}]
Let $L/K$ be an extension over which $E$ has good reduction. By Corollary \ref{cor:diffchangetogood},  we then have $\frac{\phi^* \omega^\prime_L}{\omega_L} = \frac{u^\prime}{u} \frac{\phi^* \omega ^\prime_K}{\omega_K}$ for some $u, u^\prime \in \cO_L$ with $\val_L (u) = \frac{e(L/K)\vm(E/K)}{12}$ and  $\val_L (u^\prime) =\frac{e(L/K)\vm(E^\prime/K)}{12}$.

As a result, we have
\begin{equation}
\label{eq:valalphaKalphaL}
\val_L \left (\frac{\phi^* \omega^\prime_L}{\omega_L} \right ) = e(L/K) \left ( \val_K \left (\frac{\phi^* \omega^\prime_K}{\omega_K} \right) + \frac{\vm(E^\prime/K) - \vm(E/K)}{12} \right ).
\end{equation}
If $\alpha_{\phi/K} = 1$, then $\val_K \left (\frac{\phi^* \omega^\prime_K}{\omega_K} \right) = 0$, so $\val_L \left (\frac{\phi^* \omega^\prime_L}{\omega_L} \right ) = e(L/K)\frac{\vm(E^\prime/K) - \vm(E/K)}{12}$. However, since $E/L$ has supersingular reduction, we know by Proposition \ref{prop:ssbothbad} that $\val_L \left (\frac{\phi^* \omega^\prime_L}{\omega_L} \right )> 0$. As a result, we must have $\vm(E/K) < \vm(E^\prime/K)$. The fact that $m(E/K) < m(E^\prime/K)$ then follows from Ogg's formula as explained in Remark \ref{rem:Ogg}. This proves (i).

To prove (ii), we observe that if %$\alpha_{\phi/K} = \N_{K/\QQ_p}(p)$,
$\alpha_{\phi/K} = p^{\deg K}$, then by Corollary \ref{cor:sumtovalp}, we must have $\alpha_{\phi^\prime/K} = 1$. Exchanging the roles of $E$ and $E^\prime$ and applying (i) then yields the result.
\end{proof}

Theorem \ref{thm:mainthm} now follows almost immediately.
\begin{proof}[Proof of Theorem \ref{thm:mainthm}]
By Corollary \ref{cor:oneeq1}, one of $\alpha_{\phi/K}$ and $\alpha_{\phi^\prime/K}$ is equal to $1$ and the other is equal to %$\N_{K/\QQ_p}(p)$.
$p^{\deg K}$. Theorem \ref{thm:mainthm} then follows from Theorem \ref{thm:ramthm}.
\end{proof}

\end{document}